\documentclass[11pt,twoside]{amsart}
\usepackage[all]{xy}
        \usepackage {amssymb,latexsym,amsthm,amsmath,mathtools,mathrsfs}
        \usepackage{enumitem,color}
        
\usepackage{array}

\allowdisplaybreaks[4]
\usepackage{hyperref}
\usepackage[capitalize,nameinlink]{cleveref} 
\usepackage{xcolor} 
\hypersetup{
    colorlinks,
    linkcolor={red!80!black},
    citecolor={green!80!black},
    urlcolor={blue!80!black}
}
\usepackage[maxbibnames=99]{biblatex} 
\usepackage{mathtools}

\long\def\symbolfootnote[#1]#2{\begingroup%
\def\thefootnote{\fnsymbol{footnote}}\footnote[#1]{#2}\endgroup}

\newcommand{\GL}{\mathrm{GL}}

\newcommand{\M}{\mathfrak{gl}}
\newcommand{\R}{\mathrm{R}}

\makeatletter
\def\imod#1{\allowbreak\mkern10mu({\operator@font mod}\,\,#1)}
\makeatother

\newtheorem{theorem}{Theorem}[section]

\newtheorem*{theorem*}{Theorem}
\theoremstyle{definition}

\newtheorem{remark}[theorem]{Remark}

\numberwithin{equation}{section}

\newcommand{\ignore}[1]{}

\newcommand{\mynote}[1]{}

\usepackage{parskip}
\begin{document}
\setcounter{section}{0}
\title[Word maps, polynomial maps and image ratios]{Word maps, polynomial maps and image ratios}
\author{Saikat Panja}
\email{panjasaikat300@gmail.com}
\address{Harish-Chandra Research Institute- Main Building, Chhatnag Rd, Jhusi, Uttar Pradesh 211019, India}
\thanks{The author is supported by a PDF-Math fellowship from the Harish-Chandra Research Institute}
\date{\today}
\subjclass[2020]{20G40, 20P05, 16R10, 16S50}
\keywords{word maps, polynomial maps, finite groups, finite rings}
\begin{abstract}
If $A$ is a finite group (or a finite ring) and $\omega$ is a word map (or a polynomial map), we define the quantity $|\omega(A)|/|A|$ as the \emph{image ratio of $\omega$ on $A$} and will be denoted by $\mu(\omega,A)$. In this article, we investigate the set $\mathrm{R}(\omega)=\{\mu(\omega,A) : A \text{ is a finite group}\}$, and also consider the case of rings. Specifically, we demonstrate the existence of word maps (and polynomial maps) whose set of image ratios is dense in $[0,1]$ for groups (and rings).
\end{abstract}
\maketitle
\section{Introduction}\label{sec:introin}
Let $\mathbf{C}$ be the category of finite groups or algebras.
An element $\omega$ from a finitely generated free group or free polynomial ring will be referred to as \emph{word}. 
For $A\in \mathbf{C}$, and a word $\omega$ on $n$ generators, $\omega$ induces a set-theoretic map
\begin{align}\label{eq:1}
    \widetilde{\omega}:A^n\longrightarrow A,
\end{align}
by means of evaluation. The image will be denoted by $\omega(A)$, by abuse of notation. In the case of groups, these are known as \emph{word maps} and in the case of algebras they are known as \emph{polynomial maps}, however, we will stick with the term word map throughout the article.
After the settlement of Ore's conjecture in \cite{LOST10}, there has been a growing interest in the direction of word maps in groups. 
The study of the polynomial maps was revamped after a somewhat positive solution in \cite{KaMaRo12} to a very old conjecture by L'vov and Kaplansky in the case of a quadratically closed field $F$ and $\M_2(F)$, the full matrix ring of $2\times 2$ matrices with entries from $F$. 
These results incorporate results from the theory of algebraic groups, representation theory of groups, \emph{etc.} to handle the problems in group theory whereas results from algebraic geometry, arithmetic properties of polynomial equations in associative algebras \emph{etc.} are being used to obtain results in case of algebras.
An interesting related problem in the theory of word maps (and polynomial maps) is a \emph{Waring-type} problem which asks given a group/algebra and a group $G$. 
There are several results in different directions and of different kinds.
We mention a few of them; in a series of three papers \cite{LiebeckShalev01}, \cite{Shalev09} and \cite{LST11} it was proved that for any finite nonabelian simple group $G$ of sufficiently high order and a word $\omega$, it satisfies $\omega(G)^2=G$. 
After \cite{KaMaRo12} appeared there has been results on $3\times 3$ full matrix ring over algebraically closed fields (see \cite{KaMaRo16}), on upper-triangular matrices in \cite{GargateMello22}, \cite{PanjaPrasad23}, for matrix ring over reals or quaternions in \cite{PaSaSi24Surj} and on more general algebras in \cite{Matej20}, \cite{MatejPeter23}. 
There are several generalizations of these concepts as well, for example, word maps with constants (see \cite{GoKuPlo18}) or polynomial maps with constants (see \cite{PaSaSi23}).
Given that the literature in this area has grown so much, it is not possible to include all of them, hence we suggest the readers to the articles \cite{Shalev13} for the group theoretic aspect and \cite{KaMaRoYa20} for the ring theoretic aspects.
 
When $A$ in \cref{eq:1} is finite, it makes sense to talk about the quantity $|\omega(A)|/|A|$.
For a fixed $A$ and $\omega$, this will be referred to as \emph{image ratio of $\omega$ on $A$} and to be denoted by $\mu(\omega,A)$. 
Clearly this ratio satisfies $\dfrac{1}{|A|}\leq \mu(\omega,A)\leq 1$.
We are interested in studying the set
\begin{align*}
    \mathrm{R}_{\mathbf{C}}(\omega)=\{\mu(\omega,A):A\in\mathbf{C}\}.
\end{align*}
More precisely we are interested in the limit points of the set $\mathrm{R}_{\mathbf{C}}(\omega)$.
Note that this definition can easily be adapted to the category where we have a notion of measure and all objects have finite measure, for example, the category of compact groups. 
In this article when considering algebras we will restrict ourselves to the category of finite $\mathbb{F}_q$-algebras, where $\mathbb{F}_q$ is a finite field of cardinality $q$.
In \cref{sec:finite-groups} we examine a few words and their possible limit points.
Furthermore, we show that if $\mathbf{C}$ is the category of finite groups then there exist power maps corresponding to $M\geq 2$, such that $ \overline{\mathrm{R}_{\mathbf{C}}(\omega)}=[0,1]$, see \cref{thm:main-group}.
Although the case for finite rings will follow from this, we mention them in \cref{sec:ratio-ring} separately. An analogous result to \cref{thm:main-group}, is stated without proof in this section, see \cref{thm:main-ring}.
\section{Image ratios of power maps in finite groups}\label{sec:finite-groups}
Let $F_n$ denote the free group on $n$ generators. If $\omega\in [F_n,F_n]$ is of the form
\begin{align*}
    \omega=\prod\limits_{i=1}^{t}x_{m_i}^{s_i},
\end{align*}
then for each $x_j$ the total power of $x_j$ in $\omega$ is $0$. 
Hence $\mu(\omega, \mathbb{Z}/n\mathbb{Z})=\dfrac{1}{n}$, which tends to $0$ if $n\rightarrow\infty$. This shows that
    if $\omega\in [F_n,F_n]$, then $0\in \overline{\R(\omega)}$.

If we take $\omega\in F_n\setminus [F_n,F_n]$ then there exist integers $a_1,a_2,\ldots,a_n\geq 0$ and $\omega'\in [F_n,F_n]$ such that
$$\omega=x_1^{a_1}x_2^{a_2}\ldots x_n^{a_n}\omega'.$$
If $\mathrm{g.c.d.}(a_1,a_2\ldots, a_n)=a>1$ then $\mu(\omega,(\mathbb{Z}/k\mathbb{Z})^t)=\dfrac{1}{k^t}$, which proves that in this case also $0\in \overline{\R(\omega)}$.
We now work with a special family of words known as Engel words.
These are words of $[F_2,F_2]$, defined inductively as follows;
\begin{align*}
    e_1&=[x,y]=xyx^{-1}y^{-1},\\
    e_i&=[e_{i-1},y],\text{for all }i>1.
\end{align*}
It follows from \cite[Theorem A]{Bandman12} that for $i\in\mathbb{N}$ there exists $q_0(i)$ such that for all prime power $q\geq q_0(i)$ the $i$-th Engel word $e_i$ satisfies $e_i(\mathrm{SL}_2(q))=\mathrm{SL}_2(q)\setminus\{\mathrm{Id}_2\}$. 
Hence we get 
\begin{remark}
For $i\geq 1$ the $i$-th Engel word satisfies $0,1\in\overline{\R(e_i)}$.    
\end{remark}
Now we examine the power maps which are induced by elements of $F_1$ and given by $\theta_M: G\longrightarrow G$, given by $\theta_M(x)=x^M$.
These maps are interesting in their own right; they produce one of the examples of words $\omega$ such that a finite non-abelian simple group $G$ may have $\omega(G)^2\neq G$.
This is why the results of \cite{LST11} are optimal.
These maps have been studied for several groups, notably in a series of papers for several finite groups of Lie type, see \cite{KunduSingh22}, \cite{PanjaSinghSympOrth22} and \cite{PanjaSinghUnitary23}. 
They have interesting applications as well, for example, see \cite{panja2023roots}, \cite{panja2024fibers}.
Let $p$ be a prime. 
Consider the group $\GL_2(q)$, where $q=p^r$ for some $r>0$.
Also, assume $M$ is a positive integer such that $p|M$ and $(q^2-1,M)=1$.
We wish to find $\mu(x^M,\GL_2(q))$ in this case. 
It is well-known that the conjugacy classes of $\GL_2(q)$ are given by the representatives
\begin{align*}
    \left\{\begin{pmatrix}
        \lambda&0\\0&\lambda
    \end{pmatrix}:\lambda\in \mathbb{F}_{q}^\times\right\},&
    \left\{\begin{pmatrix}
        \lambda&0\\0&\mu
    \end{pmatrix}:\lambda,\mu\in \mathbb{F}_{q}^\times,\lambda\neq \mu\right\},,\\
    \left\{\begin{pmatrix}
        \lambda&1\\0&\lambda
    \end{pmatrix}:\lambda\in \mathbb{F}_{q}^\times\right\},&
    \left\{\begin{pmatrix}
        0&\beta\\1&-\alpha
    \end{pmatrix}:\substack{\alpha,\beta\in \mathbb{F}_{q},\\t^2-t\alpha+\beta\text{ is }\\\text{irreducible in }\mathbb{F}_{q}[t]}\right\}.
\end{align*}
While considering the map $\theta_M:\GL_2(q)\longrightarrow\GL_2(q)$ given by $\theta_M(g)=g^M$, we get that the conjugacy class of elements of the form $\begin{pmatrix}
    \lambda & 1\\
    0 & \lambda
\end{pmatrix}$ do not survive.
Hence we get that 
\begin{align*}
    |\theta_M\left(\GL_2(q)\right)|=(q^2-1)(q^2-q)- (q-1)(q^2-1),
\end{align*}
which in turn implies that $\mu(x^M,\GL_2(q))=1-\dfrac{1}{q}$.
Hence combined with the previous result we get that 
if $M$ is a power of a prime $p$, both $1$ and $0$ are limit points of $\R(\theta_M)$ in $[0,1]$ where $\omega=x^M$. 
We claim that this is the case for any $M$, whenever $6\nmid M$, in particular for odd $M$ or even $M$ with $3\nmid M$.

Assume $6\nmid M$ and $M=p_1^{t_1}p_2^{t_2}\ldots p_m^{t_m}$ be a prime factorization of $M$ into distinct primes satisfying $p_1<p_2<\ldots<p_m$.
Then $p_1^2\not\equiv 1\pmod{p_i}$ for all $2\leq i\leq m$. Let $b_i>1$ be the least positive integer satisfying
\begin{align*}
    p_1^{2b_i}\equiv 1\pmod{p_i},
\end{align*}
for all $2\leq i\leq m$. Hence $p_i|p_1^{2t}-1$ for all $2\leq i\leq m$ if and only if $t$ is divisible by the $b=\mathrm{l.c.m.}(b_2,b_3,\ldots,b_m)>1$. 
This implies that there exist infinitely many $r$, say $r_1,r_2,\ldots$ such that $(p_1^{2r_s}-1,M)=1$. 
We use these $r_i$'s and note that $\mu(x^M,\GL_2(p_1^{r_i}))=1-\dfrac{1}{p_1^{r_i}}$ which tends to $1$ as $i\rightarrow\infty$. Thus we have proved
that if $M\geq 2$ is an integer such that $6\nmid M$, then $0,1\in\overline{\R(x^M)}$. We prove a more general statement below;
\begin{theorem}\label{thm:main-group}
    Let $M=2^a$ for an integer $a\geq 1$. Then $\overline{\R(x^M)}=[0,1]$.
\end{theorem}
\begin{proof}
For two finite groups $H$ and $K$ and any word map $\omega$ we have $$\mu(\omega, H\times K)=\mu(\omega, H)\times \mu(\omega, K).$$ 
From previous discussion, we know that there are groups $H_i$ such that $\mu(x^M,H_i)=1-\dfrac{1}{2^{i}}$. 
Also considering the groups $H_j=(\mathbb{Z}/2M\mathbb{Z})^j$, the $j$-fold direct product of the cyclic groups, for all integer $j\geq 1$ we get that $\mu(x^M,H_j)=\dfrac{1}{2^j}$. 
The rest of the proof relies on the following partitions of intervals;
\begin{align*}
    (0,1)=\bigcup\limits_{t\geq 0}\left[\dfrac{1}{2^{t+1}},\dfrac{1}{2^{t}}\right),\,
    \left[\dfrac{1}{2},1\right)=\bigcup\limits_{t\geq 1}\left[1-\dfrac{1}{2^t}, 1-\dfrac{1}{2^{t+1}}\right).
\end{align*}
Fix an element $c\in (0,1)$. 
Then there exists $m\geq 1$ such that $\dfrac{1}{2}\leq d=2^mc<1$. 
Hence there exists $n_1$ such that
\begin{align*}
    u_1=1-\dfrac{1}{2^{n_1}}\leq d< 1-\dfrac{1}{2^{n_1+1}}=v_1,
\end{align*}
which implies 
\begin{align*}
    \dfrac{1}{2}\leq \dfrac{u_1}{v_1}\leq \dfrac{d}{v_1}<1.
\end{align*}
Then there exists $n_2$ such that $u_2=1-\dfrac{1}{2^{n_2}}\leq \dfrac{d}{v_1}<1-\dfrac{1}{2^{n_2+1}}=v_2$ and so on. 
This process produces three sequences; (1) a sequence $\{n_i\}$ of natural numbers, (2) and two sequences $\{u_i\}$ and $\{v_i\}$ of rationals.
It can be shown that (1) the sequence $\{n_i\}$ is non-decreasing and (2) the sequence $\{u_i\}$ is non-decreasing and has infinitely many distinct terms.
Since the sequence $\{u_i\}$ is bounded above by $1$ it follows that it converses to $1$ and hence
\begin{align*}
    \lim\limits_{i\longrightarrow\infty}\dfrac{d}{v_1v_2\ldots v_{i-1}}=1,
\end{align*}
and thus $\dfrac{v_1v_2\ldots v_{i-1}}{2^m} \longrightarrow c$ as $i\longrightarrow\infty$. 
It is clear that using the groups $G_i$ and $H_j$ each of the term
\begin{align*}
    \dfrac{1}{2^{m_0}}\left(1-\dfrac{1}{2^{m_1}}\right)\ldots \left(1-\dfrac{1}{2^{m_k}}\right)
\end{align*}
are achievable for $m_j\geq 1$ and $k\geq 0$. This finishes the proof.
\end{proof}

\section{Image ratios of power maps in finite rings}\label{sec:ratio-ring}
A similar treatment applies to polynomial maps as well and hence we keep this section very short by just producing one map with a dense set of image ratios.
We will show that the squaring map corresponding to the polynomial $\omega=x^t\in \mathbb{F}_q[t]$ satisfies $\overline{\R(\omega)}=[0,1]$, when the category is taken to be of finite algebras over a finite field of cardinality $2$.
We work with the full matrix ring $\M_2(2^r)$. 
From the theory of Jordan canonical form it follows that the conjugacy classes of this ring fall into the following classes;
\begin{align*}
    \left\{\begin{pmatrix}
        \lambda&0\\0&\lambda
    \end{pmatrix}:\lambda\in \mathbb{F}_{2^r}\right\},&
    \left\{\begin{pmatrix}
        \lambda&0\\0&\mu
    \end{pmatrix}:\lambda,\mu\in \mathbb{F}_{2^r},\lambda\neq \mu\right\},,\\
    \left\{\begin{pmatrix}
        \lambda&1\\0&\lambda
    \end{pmatrix}:\lambda\in \mathbb{F}_{2^r}\right\},&
    \left\{\begin{pmatrix}
        0&\beta\\1&-\alpha
    \end{pmatrix}:\substack{\alpha,\beta\in \mathbb{F}_{2^r},\\t^2-t\alpha+\beta\text{ is }\\\text{irreducible in }\mathbb{F}_{2^r}[t]}\right\}.
\end{align*}
Considering the squaring map on $\M_2(2^r)$, we note that all but the elements of the form $\begin{pmatrix}
    \lambda & 1\\0&\lambda 
\end{pmatrix}$ appear in the image. Hence considering the conjugacy class sizes, we get that
\begin{align*}
    \mu(x^2,\M_2(2^r))=1-\dfrac{1}{2^{r}}.
\end{align*}
Since $\mu(x^2,(\mathbb{Z}/4\mathbb{Z})^n)=\dfrac{1}{2^n}$, it follows from the arguments of \cref{sec:finite-groups} that $\overline{\R(\omega)}=[0,1]$ and we tally that in the following theorem at this end.
\begin{theorem}\label{thm:main-ring}
    Let $M=2^a$ for some integer $a\geq 2$. When considered $x^M$ as a polynomial map for the category of $\mathbb{F}_{2}$-algebras, we have $\overline{\R(x^M)}=[0,1]$.
\end{theorem}
\printbibliography
\end{document}